\documentclass[12pt,twoside,reqno]{amsart}
\usepackage{amsmath, amsthm, amscd, amsfonts, amssymb, graphicx}
\usepackage[utf8]{inputenc}
\usepackage{cleveref}
\crefname{section}{§}{§§}
\Crefname{section}{§}{§§}
\let
\oldsection
\section
\renewcommand{\section}{\vspace{8pt plus 4pt}\oldsection}
\usepackage{lipsum}
\usepackage[super]{nth}\usepackage{xcolor}
\usepackage{tikz}
\usepackage{pgfplots}
\usepackage{mathrsfs}
\usepackage{graphics}
\usepackage{graphicx} 
\usepackage{epsfig}
\usepackage{wrapfig}

\newtheorem{thm}{Theorem}[section]

\newtheorem{corollary}[thm]{Corollary}
\newtheorem{lemma}[thm]{Lemma}
\newtheorem{prop}[thm]{Proposition}

\theoremstyle{definition}
\newtheorem{defn}{Definition}[section]

\theoremstyle{remark}
\newtheorem{rem}{Remark}[section]

\numberwithin{equation}{section}
\begin{document}
\begin{center}\large{{\bf{  Countable additivity of Henstock-Dunford Integral and Orlicz Space}}} 
\vspace{0.5cm}

Hemanta Kalita$^{1\ast}$ and Bipan Hazarika$^{2}$ 

\vspace{0.5cm}
$^{1}$Department of Mathematics, Patkai Christian College (Autonomous), Dimapur, Patkai 797103, Nagaland, India\\
$^{2}$Department of Mathematics, Gauhati University, Guwahati 781014, Assam, India\\
 
Email: hemanta30kalita@gmail.com; bh\_rgu@yahoo.co.in 
\thanks{$^{\ast}$ The corresponding author.}
\end{center}
\title{}
\author{}
\thanks{}
\date{\today}
\begin{abstract} Given a real Banach space $\mathcal{X}$ and probability space $(\Omega, \Sigma, \mu)$ we characterize the countable additivity of Henstock-Dunford integral for Henstock integrable function taking values in $X$ as those weakly measurable function $ g: \Omega \to \mathcal{X} $ for which $\{y^*g~: y^* \in B_\mathcal{X}^* \} $ is relatively weakly compact in some separable Orlicz space $ H^\phi(\mu) .$  We find  relatively weakly compact in some Orlicz space with Henstock-Gel'fand integral.\\
\parindent=5mm
\noindent{\footnotesize {\bf{Keywords and phrases:}}}   Henstock-Pettis integral; Denjoy-Dunford integral; Orlicz Space.\\
{\footnotesize {\bf{AMS subject classification \textrm{(2010)}:}}} 26A39, 26A42, 28B05, 46E30, 46G10.
\end{abstract}
\maketitle

\maketitle

\pagestyle{myheadings}
\markboth{\rightline {\scriptsize     Kalita and Hazarika}}
        {\leftline{\scriptsize  }}

\maketitle
\section{Introduction and Preliminaries}
    During 1957-1958, R. Henstock and J. Kurzweil independently gave a Riemann type integral called Henstock-Kurzweil integral (or Henstock integral. Let $\mathcal{I}_0$ be a compact interval in $ \mathbb{R}^m$ (or $\mathbb{R}^1$) and $ \mathcal{E} \subset \mathbb{R}^m$ (or $\mathbb{R}$) a measurable subset of $\mathcal{ I}_0.$ $\mu(\mathcal{\mathcal{E}})$ stands for the Lebesgue measure. The Lebesgue integral of a function $ g$ over set $ \mathcal{E} $ will denoted by $ \mathcal{L}\int\limits_\mathcal{E} g.$  $ \mathcal{X}$ is a real Banach space with norm $||.|| $ and $\mathcal{X}^* $ is its dual. $ B_\mathcal{X}^*= \{ y^* \in \mathcal{X} : ||y^*|| \leq 1 \} $ is the closed unit ball in $ \mathcal{X}^*.$   Henstock integral (see \cite{Ye}) is a kind of non absolute integral and contain Lebesgue integral.Recalling Measurable integral are Henstrock integrable, are lebesgue integrable. In \cite{BH} we found Orlicz space with measurable Henstock integral  $H^\phi(\mu)$  is  equivalent to the  classical Orlicz space. It has been proved that this integral is equivalent to the special Denjoy integral, also in \cite{TG} we found Henstock integral is equivalent to  Denjoy integral. In \cite{R.A} Gordon gave two Denjoy-type extensions of the Dunford and Pettis integrals, the Denjoy-Dunford and Denjoy-Pettis integrals, and discuss their properties. In \cite{Ye} authors  discussed  the relationship between Henstock-Dunford and Henstock-Pettis integral. The de la Vallee-Poussin theorem (VPT) is use in \cite{Ricardo} to localization of uniformly integrable subset of scalar integrable function( Dunford integral)  with respect to a vector measure `$m$' in a suitable Orlicz space. Also one can see \cite{Alex} for de la Vallee-Poussin theorem and Orlicz spaces. In \cite{D} author 
   characterized the countable additivity of the Dunford integral of vector functions and also they characterize those strongly measurable vector function that are Pettis integrable through the compactness of certain set of scalar function in a certain Orlicz space. The Dunford , Pettis and Gel’fand integrals are generalizations of the Lebesgue
integral to Banach-valued functions. On the other hand, the Henstock integral of
a scalar function is a kind of nonabsolute integral which generalizes the Lebesgue
integral. Therefore, replacing the role of the Lebesgue integral by the Henstock
integral in those integrals, we obtain the Henstock-Dunford , Henstock-Pettis
and Henstock-Gel’fand integrals (which are extensions of the Dunford, Pettis
and Gel’fand integrals, respectively). \\
This paper is an attempt to characterize
the countable additivity of the Henstock-Dunford integral with the help of VPT (\cite{M}, Theorem 2. p3) and Dunford Pettis theorem, we find countable additivity of Henstock-Dunford integral of Henstock integrable function taking values in $\mathcal{X}$ as those weakly measurable function $ g:[a,b] \to \mathcal{X} $ for which $ \{ y^* g : y^* \in B_{\mathcal{X}}^{*} \} $ is relatively weakly compact in some separable Orlicz space $ H^\phi(\mu) $
     Also we execute necessary condition of Henstock-Gel'fand in terms of relatively weakly in $ H^\phi(\mu) $ Lastly we find $ \{ y^* g : y^* \in B_{\mathcal{X}}^{*} \} $ relatively weakly compact in some separable Orlicz space $ H^\phi(\mu) $ of weakly Henstock integrable function.  In \cite{BH} mention $H(I)$ is equivalent $L^1(.)$ with this consideration where $H(I) $ is Henstock integrable function space and $L^1(.)$ is classical lebesgue space. \\
   The intervals $ I $ and $J$ are non overlapping if int$(I) \cap $int$(J) = \phi ,$ where int$(I),$ int$(J)$ is interior of $I$ and $J,$ respectively.\\
   Now we recall some definitions and notions used in \cite{Kao}.\\
      Let $ P$ be a partition of the interval $[a,b]$ with $ P=a=y_0 < y_1 < y_2...<y_n=b.$ A tagged partition $ (P,(v_k)_{k=1}^{n} $ is a partition which has selected points $a_k$ in each subinterval $[y_{k-1}, y_k].$ The Riemann sum using the tagged partition can be written $$ \mathcal{R}(g,P)= \sum_{k=1}^{n}g(v_k)[y_k, y_{k-1}] .$$
    Let $ \delta > 0 .$ A partition $P$ is $\delta$-fine if every sub interval $ [y_{k-1}, y_k] $ satisfies $ y_k - y_{k-1} < \delta .$\\
    A function $ \delta: [a,b] \to \mathbb{R} $ is called a gauge on $[a,b]$ if $\delta(y) > 0 $ for all $ y \in [a,b ] .$ \\
     For example of $\delta(y)$-fine tagged $P$ partition. 
          Consider the interval $[0,1]$ and $\delta_1(y)=\frac{1}{8}, $ we will find a $\delta_1(y) $ fine tagged partition on $[0,1] .$\\
     For the choice of tag, $ \delta_1(v_k)=\frac{1}{8} $ any tagged partition $(P, (v_k)_{k=1}^{n})$ in which $ y_k - y_{k-1} < \frac{1}{8} $ is a $\delta_1(y) $ fine tagged partition.\\
     Consider the following partition, choosing each tag from every interval to be any number in that interval:\\
     $m([0, \frac{1}{9}]) < \frac{1}{8},~m([\frac{1}{9}, \frac{2}{9}]) < \frac{1}{8},~m([\frac{2}{9}, \frac{3}{9}]) < \frac{1}{8},~m([\frac{3}{9}, \frac{4}{9}]) < \frac{1}{8},~m([\frac{4}{9}, \frac{5}{9}]) < \frac{1}{8}, ~m([\frac{5}{9}, \frac{6}{9}]) < \frac{1}{8}, ~m([\frac{6}{9}, \frac{7}{9}]) < \frac{1}{8},~m([\frac{7}{9}, \frac{8}{9}]) < \frac{1}{8},~m([\frac{8}{9}, \frac{9}{9}]) < \frac{1}{8} $ is an example of a $ \delta_1(y) $ fine tagged.     
   \begin{defn}\cite{Kao}
   A function $ g:[a,b] \to \mathbb{R} $ is Henstock integral if there exists $ A \in \mathbb{R} $ such that for $ \epsilon > 0 $ there exist a gauge $\delta:[a,b] \to \mathbb{R} $ such that for each tagged partition $(P, (v_k)_{k=1}^{n})$ that is $\delta(y)$ fine, $$|\mathcal{R}(g,P)-A| < \epsilon $$ Or A function $ g:[a,b] \to \mathbb{R} $ is Henstock integrable if there exists a function $ G:[a,b] \to \mathbb{R} $ such that for every $ \epsilon > 0 $ there is a function $ \delta(t) > 0 $ such that for any $ \delta-$fine partition $ D=\{[u,z], t \} $ of $[a,b], $ we have $$|| \sum[g(t)(z-u)-G(u,z)]|| < \epsilon $$ where the sum $ \sum $ is understood to be over $ D= \{ ([u,z], t) \}$ and $G(u,z)= G(z)-G(u).$     We write $ \mathcal{H}\int\limits_{\mathcal{I}_0} g=G(\mathcal{I}_0).$
   \end{defn}
 \begin{defn}
 \begin{enumerate}
 \item[(a)] \cite{JD} A function $ g:[a,b] \to \mathcal{X} $ is said to be Dunford integrable on $ [a,b] $ if for each $ y^* \in \mathcal{X}^*, $ the function $ y^*g$ is Lebesgue integrable. In this case, as a consequence of the closed graph theorem, for every measurable subset $ A $ of $[a,b],$ there exists a vector $ y_{A}^{**}$ in $ \mathcal{X}^{**}$ such that $$ < y^* , y_{A}^{**} > = \int\limits_A y^* g \mbox{~for~all~}y^* \in \mathcal{X}^* .$$  A vector $ {y_A}^{**}$ is called the Dunford integral of $ g $ on $ A$ and is denoted by $ \mathcal{D}\int\limits_A g.$
 \item[(b)] \cite{JD} A function $ g:[a,b] \to \mathcal{X} $ is said to be Pettis integrable on $ [a,b] $ if it is Dunford integrable on $[a,b]$ and $ y_{A}^{**} \in \mathcal{X} $ for every measurable subset $ A $ of $[a,b]$\\
   The Henstock integral to Banach valued functions, is exactly in the same way as the Dunford and Pettis are extensions of the Lebesgue integral.  we define Henstock-Gel'fand integral as the style of R.A. Gordon \cite{Fong} of his Denjoy-Dunford and Denjoy-Pettis integral. Also we refer \cite{N.Dunford,Geo,Morrison,S} for the reader related to this areas. 
 \end{enumerate}
 \end{defn}
 
 \begin{defn}
 \begin{enumerate}
 \item[(a)] \cite{Ye} A function $ g:[a,b] \to \mathcal{X} $ is said to be Henstock-Dunford integrable on $[a,b]$ if for each $ y^* $ in $ \mathcal{X}^* ,$ the function $ y^* g $ is Henstock integrable on $[a,b]$ and if for every interval $ \mathcal{I} $ in $[a,b], $ there exists a vector $ y_\mathcal{I}^{**} $ in $ \mathcal{X}^{**} $ such that $$  y_\mathcal{I}^{**}(y^*) = \int\limits_{\mathcal{I}} y^* g \mbox{~for~all~}  y^* \in \mathcal{X}^* .$$ We write $ y_{\mathcal{I}_0}^{**} = \mathcal{HD} \int\limits_{\mathcal{I}_0} g = G(\mathcal{I}_0).$
  \item[(b)]\cite{Ye}  A function $ g:[a,b] \to \mathcal{X} $ is said to be Henstock-Pettis integrable on $[a,b]$ if $g $ is Henstock-Dunford integrable on $[a,b]$ and if $ y_{\mathcal{I}}^{**} $ in $ \mathcal{X} $ for every interval $\mathcal{ I} $ in $[a,b].$ We write $$ y_{\mathcal{I}_0}^{**} = \mathcal{HP} \int\limits_{\mathcal{I}_0} g = G(\mathcal{I}_0).$$
 \end{enumerate}
 \end{defn}
 \begin{defn} 
 \begin{enumerate}
 \item[(a)] \cite{JD} A function $ g:[a,b] \to \mathcal{X}^* $ is said to be Gel'fand integrable on $[a,b]$ if for each $ y \in \mathcal{X}, yg$ is Lebesgue integrable. In this as consequence of the closed graph theorem, for every measurable subset $ A $ of $[a,b]$ there exist $ y_{A}^{*} $ in $ \mathcal{X}^* $ such that $$ < y_{A}^{*} , y > = \int\limits_A yg \mbox{~for~all~}  y^* \in \mathcal{X}^* .$$   The vector $ y_{A}^{*} $ is called the Gel'fand integral of $ g$ on $ A $ and is denoted by $ \mathcal{G}\int\limits_A g $
   \item[(b)] \cite{B. Bongiorno}  A function $ g:[a,b] \to \mathcal{X}^* $ is said to be Henstock-Gel'fand integrable on $[a,b]$ if for each $ y \in \mathcal{X}, yg$ is Henstock integrable on $[a,b] $ and for every interval $ \mathcal{I} $ in $[a,b] $ there exist a vector $y_{\mathcal{I}}^{*} \in \mathcal{X} ^* $ such that $ y^*(y) = \int\limits_{\mathcal{I}} yg.$ 
 \end{enumerate}
 \end{defn} 
  \begin{defn} \cite{Mohammed} A function $ g:[a,b] \to \mathcal{X} $ is said to be weakly Henstock integrable $(w\mathcal{H})$ on $[a,b]$ with weak integral $\overline{w},$ if there is a sequence of gauges $(\delta_n)$ on $[a,b]$ such that $$ \lim\limits_{n \to \infty }< y^* , \sigma(g, p_n)> = < y^* ,\overline{w}> \mbox{~for~all~}  y^* \in \mathcal{X}^* .$$  For every sequence $(p_n) $ of Henstock integrable partition of $[a,b]$ adapted to $(\delta_n)$ and $ \overline{w} = \left((w\mathcal{H})-\int\limits_{a}^{b}g\right).$
  \end{defn} 
  But for our work we prefer the definition of Weakly Henstock-integrable like  Y. Guoju,S. Schwabik define weakly Mcshane integral in \cite{SI}, so we state our definition
  \begin{defn}
  A function $ g:[a,b] \to \mathcal{X} $ is said to be weakly Henstock integrable $(w\mathcal{H})$ on $[a,b]$ if for every $x^* \in X^* $ the real function $x^*(f) $ is Henstock integrable on $[a,b]$ and for every interval $I \subset [a,b],$ there is a $x_I \in \mathcal{X}$ such that $\int_{I}x^*(f)=x^*(x_I),$ \\
   We write $(w\mathcal{H})\int_I f = X_I$
  \end{defn}
  \begin{defn}
  A function $g:[a,b] \to \mathcal{X} $ is said to be weakly Henstock -Dunford integrable $(w\mathcal{HD})$ on $[a,b]$ if for every $x^* \in X^* $ the real function $x^*(f) $ is  Weak Henstock integrable on $[a,b]$ and for every interval $I \subset [a,b],$ there is a $x_{I}^{*} \in \mathcal{X}^{**}$ such that $\int_{I}x^*(f)=x_{I}^{**}(x^*),~\forall x^* \in \mathcal{X}^{*}$ 
   
  \end{defn}
  
  \begin{defn} A weakly measurable function $ g:[a,b] \to \mathcal{X} $ is said to be determined by a weakly compact generated $(WCG)$ subspace of $ \mathcal{X}, $ such that $x^* g=0 $ a.e. for all $x^* \in \mathcal{X}^*$ with $x^*|D=0$ if there is a weakly compact generated subspace $ D $ of $\mathcal{ X} $ 
  \end{defn} 
  \begin{defn} \cite{M} Let $ m : \mathbb{R}^+ \to \mathbb{R} ^+ $ be non decreasing right continuous and non negative function satisfying $$ m(0)=0, \mbox{~and~}  \lim\limits_{t \to \infty } m(t) = \infty.$$ A function $ M:\mathbb{ R} \to \mathbb{R} $ is called an $N$-function if there is a function $'m' $  satisfying the above sense that $$ M(u) = \int_{0}^{|u|} m(t) dt.$$
   Evidently, $M$ is an $N$-function if it is continuous, convex, even satisfies $$ \lim\limits_{u \to \infty } \frac{M(u)}{u}= \infty  \mbox{~and~}  \lim\limits_{u \to 0} \frac{M(u)}{u}= 0 .$$
    For example $ \overline{\phi}_p(x)= x^p;~ p > 1.$\\
     Let us fix a positive finite measure $ \mu $ and let $\overline{\phi} $ be an $N$-function. The Orlicz space $ H^\phi(\mu) $ consists those ($ \mu$-a.e. equivalence classes) of functions $ g \in H^0(\mu)$ for which $$ {H_{\phi}} = \inf \{ a >0 : \int_{\mathcal{\Omega}} \phi( \frac{f}{a}) \in H(I) \} $$
    
     \end{defn} 
   
    \begin{defn} \cite{M}
    \begin{enumerate}
    \item[(a)] An $N$-function $\overline{\phi} $ is said to satisfy $\Delta^{'} $ condition if there is a $ k > 0 $ so that $$ \overline{\phi}(xy) \leq k\overline{\phi}(x)\overline{\phi}(y) \mbox{~for~large~values~of~}x  \mbox{~and~} y.$$
    
        \item[(b)] 
          An $N$-function $\overline{\phi} $ is said to satisfy $\Delta_2 $ condition if there is a $ k > 0 $ so that $$ \overline{\phi}(2x) \leq k \overline{\phi}(x)  \mbox{~for~large~values~of~}  x.$$ 
          \end{enumerate} 
        \end{defn}  
  We recall the following results:
  \begin{thm}\label{thm12}
    A subset $ A $ of $ H^1(\mu) $ is uniformly integrable if and only if there is an $N$-function $ \phi $ with $ \Delta^{'} $ condition such that $ A $ is relatively weakly compact in $H^\phi(\mu) $ 
  \end{thm} 
  \begin{proof}
  Proof is similar as mention in \cite{Musial} for A subset $ A $ of $ \mathcal{L}^1(\mu) $ is uniformly integrable if and only if there is an $N$-function $ \phi $ with $ \Delta^{'} $ condition such that $ A $ is relatively weakly compact in $ \mathcal{L}^{\overline{\phi}}(\mu).$
  \end{proof}
     \begin{thm}\label{thm15}
     \cite{Ye}  A function  $ g:[a,b] \to \mathcal{X} $ is Henstock-Dunford integrable on $[a,b]$ if and only if $ y^* g $ is Henstock integrable on $[a,b]$ for all $ y^* \in \mathcal{X}^*.$
     \end{thm}
     \begin{lemma}\label{lemma1}
    \cite{TG} 
         A function $ g:[a,b] \to R $ is Henstock integrable on $[a,b]$ equivalent to  Denjoy integrable on $[a,b].$
  \end{lemma}
       \begin{lemma}\label{lemma1}
       \cite{Mohammed}   Suppose $\mathcal{X}$ contain no copy of $c_0$  and let $ g:[a,b] \to \mathcal{X} $ be $w\mathcal{H}$-integrable function on $[a,b],$ then it is Pettis integrable. 
       \end{lemma} 
       \begin{lemma}\label{lemm3}
       Suppose $\mathcal{X}$ contain no copy of $c_0$  and let $ g:[a,b] \to \mathcal{X} $ be $w\mathcal{HD}$-integrable function on $[a,b],$ then it is Pettis integrable.
       \end{lemma}
       \begin{proof}
       If $\mathcal{X} $ contain no copy of $C_0$ and $g:[a,b] \to \mathcal{X}$ is $w\mathcal{HD}$-integrable.  Then lemma\ref{lemma1} shows  $x^*g$ is $w\mathcal{H}$ on $[a,b]$ also $x^* g $ is measurable for all $x^* \in \mathcal{X}^{*}$\\
       That is $g $ Henstock-integrable function. so $g $ is Mc shane integrable and hence pettis integrable. 
       \end{proof}
       \begin{defn}
       \begin{enumerate}
       \item[(a)] A function $F$ is ACG on $\mathcal{E}$ if $F$ is continuous on $\mathcal{E}$ and if $E$ can be expressed as a countable union of sets on each of which $F$ is AC.
                     \item[(b)] A function  $ g:[a,b] \to \mathcal{X} $ is Denjoy integral on $[a,b]$ if there exist a ACG function $ F:[a,b] \to \mathcal{X} $ such that $ F_{ap}^{'} =g^{'} $ a.e. on $[a,b],$ where $ F_{ap}^{'}$ denotes the approximate derivatives of $F.$ In this case $$ \int\limits_{a}^{b} g = F(b) -F(a).$$ 
                     We say $g$ is Denjoy integrable on a subset $A$ of $[a,b]$ if $g\chi_A $ is Denjoy integrable on $[a,b]$ and write $$\int_A g =\int\limits_{a}^{b} g \chi_A$$
                     \item[(c)] A function $ g:[a,b] \to X^* $ is said to be Denjoy-Gel'fand integrable on $[a,b]$ if for each $ y \in \mathcal{X},$ $yg $ is Denjoy integrable on $[a,b]$ and for every interval $\mathcal{I} $ in $[a,b]$ there exists a vector $ y_{\mathcal{I}}^{*} $ is called Denjoy-Gel'fand integral of $g$ on $[a,b]$ and we denote $ \mathcal{DG}\int\limits_{a}^{b} g.$ 
       \end{enumerate}
       \end{defn}
       \begin{rem}
       (Theorem 15.9 \cite{R.A}) A real Denjoy integrable function on $[a,b]$ is not necessarily integrable on all measurable subset of $[a,b].$  In fact if the function is absolutely integrable or equivalent to Lebesgue integral.
       \end{rem}
    
       \section{Main Results} 
       
       \begin{prop}\label{prop21}
        
       Assume $ g:[a,b] \to \mathcal{X} $ is Denjoy-Gel'fand on $[a,t]$ for all $ t \in [a,b) $ and for each $ y \in \mathcal{X}. $ The limit $\lim\limits_{t \to b}\int\limits_{a}^{b} yg $ exists, then $g$ is Denjoy-Gel'fand on $[a,b]$ and $$ < y , \mathcal{DG}\int\limits_{a}^{b} g > = \lim\limits_{t \to b } <y , \mathcal{DG}\int\limits_{a}^{t} g > $$ for each $ y \in \mathcal{X}.$
       \end{prop} 
       For the proof we follow the same technique of Prop 1 of \cite{D} with minor add
     
  
  \begin{prop}\label{prop22}
  Let $ g:[a,b] \to \mathcal{X},$ if $yg $ is Denjoy integrable on $[a,b]$ for each $ y \in \mathcal{X},$ then each perfect set in $[a,b]$ contains a portion on which $g$ is Gel'fand integrable.
  \end{prop}
  For the proof of this proposition we use Theorem 33 of \cite{Gamez}, with little add 

 
  \begin{prop}\label{prop23}
  Let $ g:[a,b] \to \mathcal{X}^* $ be such that $ yg $ is Denjoy integrable on $[a,b]$ for all $ y \in \mathcal{X}.$    Let $P$ be a closed subset of $[a,b]$ and assume that $g$ is Denjoy-Gel'fand integrable on each open interval $J$ disjoint from $P,$ then there exists a portion $P_0 $ such that if $(\mathcal{I}_n)$ is an enumeration of the interval neighboring to $P_0$ then the series $ \sum\limits_{n} \int\limits_{\mathcal{I}_n} yg $ is absolutely convergent for every $ y \in \mathcal{X}.$
  \end{prop} 
 Using lemma 1 of \cite{Gamez} for the proof
  
  \begin{prop}\label{prop24} The function 
  $ g:[a,b] \to \mathcal{X}^* $ is Denjoy-Gel'fand on $[a,b]$ if and only if $yg$ is Denjoy integrable on $[a,b]$ for all $ y^* \in \mathcal{X}^*.$
  \end{prop}
  For the proof we use  Theorem 3 of \cite{Gamez}  
 
     \begin{thm}\label{thm21}
     A function $ g:[a,b] \to \mathcal{X}^* $ is Henstock-Gel'fand on $ [a,b] $ if and only if $ yg $ is Henstock integrable on $ [a,b] $
     \end{thm}
   \begin{proof}
   If $ g $ is Henstock-Gel'fand integrable on $ [a,b].$  Then by definition of Henstock-Gel'fand $yg $ is Henstock integrable on $[a,b].$\\
       Conversely, let $ yg $ be Henstock integrable on $[a,b].$ then by Lemma \ref{lemma1}, $ yg$ is Denjoy integrable on $[a,b] $ and $\mathcal{D}\int\limits_{a}^{b} yg = \mathcal{H}\int\limits_{a}^{b} yg$\\
       Then Proposition \ref{prop24} implies that $ g$ is Denjoy-Gel'fand integrable on $[a,b] $ and for every interval $ I $ on $[a,b]$ there exists a vector $ y_{I}^{*} = \mathcal{D}\int\limits_{I} yg $ for all $ y \in \mathcal{X}.$ \\
        That give us  $ y_{I}^{*} = \mathcal{H}\int\limits_{I} yg $ for all $ y \in \mathcal{X} $ so, $ g $ is Henstock-Gel'fand integrable on $[a,b].$
   \end{proof} 
     \begin{thm}
     Let $\mathcal{X}$ contain no copy of $ c_0 $ and let $ g:[a,b] \to \mathcal{X} $ is $w\mathcal{HD}$-integrable on $[a,b] $ then $\{ y^* g : y^* \in \overline{B_{\mathcal{X}}^{*}} \} $ is uniformly integrable.
          \end{thm}
     \begin{proof}
      Let $ g:[a,b] \to \mathcal{X} $ is $w\mathcal{HD}$-integrable on $[a,b] .$   Then $ g $ is  Pettis integrable.\\
          For each equi-continuous $ K \subset \mathcal{X}^* ,$ the set $\{ y^* g: y^* \in K \}$ is relatively weakly compact in $ \mathcal{L}_1(\mu).$\\
          Hence $\{ y^* g: y^* \in  \overline{B_{\mathcal{X}}^{*}} \} $ is relatively weakly compact in $ \mathcal{L}_1(\mu).$\\
          Also \cite{Ricardo} gives $\{ y^* g : y^* \in  \overline{B_{\mathcal{X}}^{*}} \} $ is uniformly integrable in $ \mathcal{L}_1(\mu) .$
          \end{proof}
\section{Henstock-Dunford Integral and Orlicz Space} 
 \begin{thm}
Let $(\Omega, \Sigma, \mu ) $ be a finite measure space and $ g: I \subseteq \Omega \to \mathcal{X} $ be Henstock-Dunford function. Then following are equivalent:
\begin{enumerate}
\item[(i)] The Henstock-Dunford integral of $ g$ is countable additive; that is the set function \\
$ \mathcal{HD}\int gd\mu: \Sigma \to \mathcal{X}^{**} $ defined $$\left(\mathcal{HD}\int gd \mu\right)(\mathcal{E})= \mathcal{HD}\int\limits_{\mathcal{E}} g d\mu$$
\item[(ii)] There is an $N$-function $\bar{\phi} $ with $ \Delta^{'} $ property such that $\{ y^* g : y^* \in B_{\mathcal{X}}^{*} \} $ is relatively weakly compact in the Orlicz space $H^\phi(\mu)$
\end{enumerate}
 \end{thm}
\begin{proof}
Let $ g:[a,b] \to \mathcal{X} $ be Henstock-Dunford integrable function. Put $ \gamma(\mathcal{E}) = \mathcal{HD}\int\limits_{\mathcal{E}}g,$ where $\mathcal{E} \in \Sigma.$\\
 Therefore by Theorem \ref{thm15}, we have   $ \gamma(E) = \mathcal{H}\int\limits_{\mathcal{E}}y^* g$\\
 Let $ E_1, E_2 $ be non overlapping subset of $ \mathcal{E}.$ Then  
 \begin{align*}
 \gamma( E_1 \cup E_2 )  &= \mathcal{HD} \int\limits_{ E_1 \cup E_2 } g\\ & = \mathcal{H} \int\limits_{E_1 \cup E_2 } y^* g \\ & = \mathcal{H} \int\limits_{E_1} y^* g +\mathcal{H} \int\limits_{E_2} y^* g \\ &= \gamma(E_1) + \gamma(E_2)
 \end{align*} 
 And  \begin{align*}
 \gamma\left(\bigcup\limits_{n=1}^{\infty}E_n\right) & = \mathcal{HD}\int\limits_{\cup_{n=1}^{\infty }E_n}g\\
 & = \mathcal{H}\int\limits_{\cup_{n=1}^{\infty }E_n}y^* g \\
 &= \sum\limits_{n=1}^{\infty} \gamma(E_n)
  \end{align*}  in norm topology of $ \mathcal{X},$ for all sequence $(E_n)$ of non overlapping members of field $ F \subseteq [a,b] $ such that $ \bigcup\limits_{n=1}^{\infty} E_n \in F $\\
  Now according as \cite{A.G} $\gamma $ is countable additive if and only if $ T :\mathcal{X}^* \to \mathcal{L}^1(\mu) $ defined by $ T(y^*) = y^* g $ is weakly compact. So, $\{ y^* g: y^* \in B_{\mathcal{X}}^{*} \} $ is uniformly integrable in $ \mathcal{L}^1(\mu) .$\\
   This will not full fill our requirement as we are considering Henstock-Dunford integral. \\
   Easily we can find $\gamma $ is countable additive if and only if $ T :\mathcal{X}^* \to H^1(\mu) $ defined by $ T(y^*) = y^* g $ is weakly compact. So, $\{ y^* g: y^* \in B_{\mathcal{X}}^{*} \} $ is uniformly integrable in $ H^\phi(\mu).$ \cite{BH} where $H^1(\mu) $ is the space of Henstock-integrable function space \\
  Now by Theorem \ref{thm12} and \cite{BH} it is equivalent to the existence of a $N$-function $\bar{\phi} $ with $ \Delta^{'} $ property such that $\{ y^* g: y^* \in B_{\mathcal{X}}^{*} \} $ is relatively compact in  $ H^\phi(\mu).$ This completes the proof.
\end{proof} 
  \begin{corollary}
  Let $ g: [a,b] \to \mathcal{X} $ be Henstock-Dunford integrable function, then followings are equivalent:
  \begin{enumerate} 
    \item[(i)] $ g$ is Henstock-Pettis integrable 
    \item[(ii)] $g $ is weakly compact generated determined and there is an $N$-function $\bar{\phi} $ with $ \Delta ^{'}$ such that $\{ y^* g: y^* \in B_{\mathcal{X}}^{*} \} $ is relatively compact in $ H^\phi(\mu) .$
    \end{enumerate}
  \end{corollary}
  \begin{thm}
    A strongly measurable function $ g:I \to \mathcal{X} ,$ and $\mathcal{X}$ contain no copy of $c_0$ then  is Henstock-Pettis integrable if and only if  there is an $N$-function $\bar{\phi} $ with $ \Delta^{'} $ property such that $\{y^* g : y^* \in B_{\mathcal{X}}^{*} \} $ is relatively weakly compact in the Orlicz space $H^\phi(\mu)$ 
    \end{thm}
  \begin{proof} If $ g:[a,b] \to \mathcal{X} $ is strongly measurable, then its range is essentially separable and weakly compact determined (Theorem 2 of \cite{JD} p.42). \\
  If $ g$ is Henstock-Pettis integrable, then each perfect set in $[a,b]$ contain a portion $P$ which is Pettis integrable (Theorem 2.6 of \cite{Ye}).  Therefore this portion is Dunford with countable additive vector measure. \\
  Hence $\{ y^* g: y^* \in B_{\mathcal{X}}^{*} \} $ is uniformly integrable on that portion and consequently, there is an $N$-function $\bar{\phi} $ with $ \Delta ^{'}$ such that $\{ y^* g: y^* \in B_{\mathcal{X}}^{*} \} $ is relatively compact in $ \mathcal{L}^{\bar{\phi}}(P) .$ So,$\{ y^* g: y^* \in B_{\mathcal{X}}^{*} \} $ is relatively compact in $H^\phi(P)$ \\
  Conversely, Suppose $ g $ is strong measurable and  there is an $N$-function $\bar{\phi} $ with $ \Delta ^{'}$ such that $\{ y^* g: y^* \in B_{\mathcal{X}}^{*} \} $ is relatively compact in $H^\phi(\mu)$
  Since $ g $ is strongly measurable, it has range weakly compactly generated determined. As $ H^\phi(\mu) \subset H^1(\mu) $ (see
  \cite{Ricardo}). So $\{ y^* g: y^* \in B_{\mathcal{X}}^{*} \} $ is a bounded subset of$ H^\phi(\mu)$ 
  As $ D_0 =\{ y^* g: y^* \in B_{\mathcal{X}}^{*} \} $ is weakly compactly generated and $ g$ is strongly measurable, So, for each $ y^* \in \mathcal{X}^* ,$ there exists a sequence $(\gamma)_{n=1}^{\infty}$ of $ D_0$-valued simple function such that $ y^*g = \lim y^* \gamma_n \mu$-a.e.\\
  So, $ g $ is Pettis integrable with $ \mathcal{X} $ contains no copy of $ c_0 $\\
  Thus  $ g$ is Henstrock-Pettis integrable with $ \mathcal{X} $ contains no copy of $ c_0 $\end{proof}
  \begin{corollary}
  Let $g :[a,b] \to \mathcal{X} $ be Henstock-Dunford. If $ p > 1 $ such that $\{ y^* g: y^* \in B_{\mathcal{X}}^{*} \} $ is bounded in $ H^\phi(\mu) $ then $ g$ is countably additive.
  \end{corollary}
  \begin{thm}
  For a class of strongly measurable function in $\mathcal{ X} , $ contains no isomorphic copy of $ c_0.$  A function $ g:[a,b] \to \mathcal{X}^* $ is Henstock-Gel'fand,  then there is an $N$-function $\bar{\phi} $ with $ \Delta ^{'}$ such that $\{ y^* g: y^* \in B_{\mathcal{X}}^{*} \} $ is relatively compact in $ H^\phi(\mu) .$
  \end{thm}
  \begin{proof}
   Let $T: \mathcal{X} \to H^1(\mu) $ by $ T(y)=yg .$ Then $T$ is bounded linear operator.\\
Claim: $T$ is weakly compact.\\
If $v(\mathcal{E})(y)= \mathcal{HG}\int\limits_{\mathcal{E}} g,$ Then $v(\mathcal{E})(y) = \mathcal{H}\int\limits_{\mathcal{E}} yg $ (by Theorem \ref{thm21})\\ 
Then $v(\mathcal{E}) $ is vector measure with countably  additivity. \\
By Pettis theorem (p10 of \cite{D}) we have  $\lim\limits_{\mu(\mathcal{E}) \to 0} v(\mathcal{E})(y)=0 $ for all $ y \in \mathcal{X} .$\\
That is for $ \epsilon > 0 $ there exists a $\delta >0 $ such that 
$$\mu(\mathcal{E})< \delta  \Rightarrow  ||v(\mathcal{E})|| < \epsilon.$$
So,  $\mu(\mathcal{E})< \delta $ implies $ \sup\limits_{y \in B_{\mathcal{X}} }\int\limits_{\mathcal{E}}|yg|d \mu < \epsilon .$\\
By Dunford-Bartle-Hanse (Corollary 6, page 14 of \cite{D}), $ v(\Sigma) $ is relatively weakly compact.\\
Therefore $ \left\{ \int\limits_{\mathcal{E}} yg : \mathcal{E} \in \Sigma , y \in B_{\mathcal{X}} \right\} $ is bounded in some field $F.$\\
So,  $ \sup\limits_{y \in B_{\mathcal{X}} }\int\limits_\Omega |yg|d \mu < \infty .$\\
Thus  $\{yg : y \in B_{\mathcal{X}} \} $ is uniformly integrable. \\
By de la Vallee Poussin theorem $\{yg : y \in B_{\mathcal{X}} \} $ is relatively weakly compact in $ H^\phi(\mu) $ with $ \Delta^{'}$ condition.
\end{proof}
\begin{thm}
For a real Banach space $\mathcal{X}$ contains no copy of $ c_0,$ if $ g:[a,b] \to \mathcal{X} $ is $w\mathcal{H}$-integrable function on $[a,b]$ then $\{ y^* g: y^* \in \overline{B_{\mathcal{X}}^{*}} \} $ is relatively weakly compact in certain separable Orlicz space $  H^\phi(\mu). $
\end{thm}
\begin{proof}  For a real Banach space $\mathcal{X}$ contains no copy of $ c_0,$ if $ g:[a,b] \to \mathcal{X} $ is $w\mathcal{H}$-integrable function on $[a,b]$ then $\{ y^* g: y^* \in \overline{ B_{\mathcal{X}}^{*}} \} $ is uniformly integrable in $  H^1(\mu).$\\
Theorem \ref{thm12} gives $\{ y^* g : y^* \in  \overline{B_{\mathcal{X}}^{*}} \} $ is relatively weakly compact in certain separable Orlicz space $ H^\phi(\mu).$\\
As $\mathcal{X}$ is reflective, it is true for $\{ y^* g : y^* \in  B_{\mathcal{X}}^{*} \}. $\end{proof}
\begin{rem}
From Lemma \ref{lemma1} countable additivity of $w\mathcal{H}$-integrable function can determine easily.
\end{rem}

\end{document}